\documentclass[12pt]{amsart}
\usepackage{amsfonts}
\usepackage{graphicx}
\usepackage{amsfonts, mathdots, yfonts, eufrak}
\usepackage{amssymb,latexsym, color}
\usepackage{enumerate}\usepackage{arydshln}\usepackage{setspace}
\usepackage{hyperref}
\makeatletter
\@namedef{subjclassname@2010}{%
  \textup{2010} Mathematics Subject Classification}
\makeatother

\newtheorem{thm}{Theorem}[section]

\newtheorem{prop}[thm]{Proposition}

\theoremstyle{definition}
\newtheorem{defin}[thm]{Definition}
\newtheorem{rem}[thm]{Remark}
\newtheorem{eg}[thm]{Example}

\numberwithin{equation}{section}
\def\tr{{\rm tr}}

\begin{document}

\baselineskip=17pt

\title[New properties]{New properties for certain positive semidefinite matrices}

\author[M. Lin]{Minghua Lin}
\address{Department of Mathematics \\
  Shanghai University\\ Shanghai, 200444, China  \\}
\email{m\_lin@i.shu.edu.cn}

\date{}


\begin{abstract}
 We bring in some new notions associated with   $2\times 2$ block positive semidefinite matrices. These notions concern the inequalities between  the  singular values of the off diagonal blocks and the eigenvalues of the arithmetic mean or geometric mean of the diagonal blocks. We investigate some relations between them. Many examples are included to illustrate these relations. 

  \end{abstract}



\maketitle

\section{Introduction}
\label{sec:intro}

  Matrices considered here have entries from the complex number field. 
We are interested in  positive  semidefinite matrices partitioned into $2\times 2$ blocks
\begin{eqnarray}\label{m}
\textbf{M}=\begin{bmatrix} M_{11} & M_{12}\\ M_{12}^* & M_{22}\end{bmatrix}.
\end{eqnarray} In particular, we assume that the off-diagonal blocks are square, say $n\times n$.

The study of eigenvalues or singular values is of  central importance in matrix analysis.   It could date back to Cauchy, who established an interlacing theorem  for a
bordered Hermitian matrix \cite[p. 242]{HJ13}. Notable results also include that of Schur \cite[p. 300]{MOA11} and Fan \cite[p. 308]{MOA11}, who revealed a majorization relation between the eigenvalues and diagonal entries (or the eigenvalues of diagonal blocks) of a Hermitian matrix. 
 For more information, we refer to \cite[Chapter III, VII, IX]{Bha97},
  \cite[Chapter 3]{HJ91}. 
  
  To proceed, let us fix some notation. For any $n\times n$ matrix $A$, the singular values $s_j(A)$ are nonincreasingly arranged, $s_1(A)\ge s_2(A)\ge \cdots \ge s_n(A)$. If $A$ is Hermitian, we also arrange its eigenvalues $\lambda_j(A)$  in nonincreasing order  $\lambda_1(A)\ge \lambda_2(A)\ge \cdots \ge \lambda_n(A)$.  The geometric mean of two $n\times n$ positive definite matrices $A$ and $B$, denoted by $A\sharp B$, is the positive definite solution of the Riccati equation $XB^{-1}X=A$ and it has the explicit expression $A\sharp B=B^{1/2}(B^{-1/2}AB^{-1/2})^{1/2}B^{1/2}$. The notion of geometric mean can be uniquely extended to
  all positive semidefinite matrices by a limit from above:
  $$A\sharp B:=\lim_{\epsilon\to 0^+}(A+\epsilon I_n)\sharp (B+\epsilon I_n),$$ where $I_n$ is the  $n\times n$ identity matrix \cite[Chapter 4]{Bha07}. For two Hermitian matrices $A$ and $B$ of the same size, we write $A\ge B$ ($A>B$) to mean that $A-B$ is positive semidefinite (positive definite).

 Now we introduce new notions to be investigated in this paper. 
 \begin{defin} Consider the matrix given in (\ref{m}), 
 	\begin{enumerate}  \item[(i)]  $\textbf{M}$ is said to have $\mathfrak{l}\mathfrak{a}$-property if\begin{eqnarray*}
 			\prod_{j=1}^k2s_j(M_{12})\le   \prod_{j=1}^k\lambda_j(M_{11}+M_{22}), \qquad k=1, \ldots, n.
 		\end{eqnarray*}
 		
 	\item[(ii)]  	$\textbf{M}$ is said to have  $\mathfrak{l}\mathfrak{g}$-property if
 		\begin{eqnarray*}
 			\prod_{j=1}^ks_j(M_{12})\le   \prod_{j=1}^k\lambda_j(M_{11}\sharp M_{22}), \qquad k=1, \ldots, n.
 		\end{eqnarray*}
 		
 	\item[(iii)]  	$\textbf{M}$ is said to have $\mathfrak{a}$-property if
 		\begin{eqnarray*}
 			2s_j(M_{12})\le \lambda_j(M_{11}+M_{22}), \qquad j=1, \ldots, n.
 		\end{eqnarray*}
 		
 	\item[(iv)]  	$\textbf{M}$ is said to have $\mathfrak{g}$-property if
 		\begin{eqnarray*}
 			s_j(M_{12})\le \lambda_j(M_{11}\sharp M_{22}), \qquad j=1, \ldots, n.
 		\end{eqnarray*} 		 
 	\end{enumerate}  	 
 \end{defin}

It is clear that $\mathfrak{g}$-property (resp.  $\mathfrak{l}\mathfrak{g}$-property) is stronger than $\mathfrak{a}$-property (resp.  $\mathfrak{l}\mathfrak{a}$-property) and   $\mathfrak{g}$-property (resp.  $\mathfrak{a}$-property) is stronger than  $\mathfrak{l}\mathfrak{g}$-property (resp.   $\mathfrak{l}\mathfrak{a}$-property). So one may conlude that   $\mathfrak{l}\mathfrak{a}$-property is the weakest one and $\mathfrak{g}$-property is the strongest one among these four properties.

But not all positive semidefinite matrices have $\mathfrak{l}\mathfrak{a}$-property.  For example, the positive  semidefinite   matrix
\begin{eqnarray*}\begin{bmatrix} M_{11} & M_{12}\\ M_{12}^* & M_{22}\end{bmatrix}=\left[\begin{array}{cc:cc} 1&0&0&1\\0&0&0&0 \\ \hdashline
 0&0&0&0 \\ 1&0&0&1 \end{array}\right]
\end{eqnarray*}  does not have  $\mathfrak{l}\mathfrak{a}$-property, as
 $$2s_1(M_{12})=2>\lambda_1(M_{11}+M_{22})=1.$$

\section{Basic observations}
Our first observation is the following
\begin{prop}\label{p1} Consider the positive  semidefinite  matrix $\textbf{M}$   given in (\ref{m}) with each block $2\times 2$. If $\textbf{M}$ has $\mathfrak{l}\mathfrak{a}$-property, then $\textbf{M}$ has either $\mathfrak{a}$-property or $\mathfrak{l}\mathfrak{g}$-property. \end{prop}
\begin{proof} As we have assumed that $\textbf{M}$ has $\mathfrak{l}\mathfrak{a}$-property, in particular,  
\begin{eqnarray}\label{pe1} s_1(M_{12})\le \frac{1}{2}\lambda_1(M_{11}+M_{22}).
\end{eqnarray} 

First of all, let us observe the following relation
\begin{eqnarray*}
\lambda_1(M_{11}\sharp M_{22})\lambda_2(M_{11}\sharp M_{22})&=&\det M_{11}\sharp M_{22}\\&=&\sqrt{\det M_{11} M_{22}}\\&\ge& |\det M_{12}|=s_1(M_{12})s_2(M_{12}),
\end{eqnarray*}
in which the inequality is due to the fact that determinant function is Liebian \cite[p. 269]{Bha97}. 
That is, 
\begin{eqnarray}\label{pe2}
	\lambda_1(M_{11}\sharp M_{22})\lambda_2(M_{11}\sharp M_{22}) \ge s_1(M_{12})s_2(M_{12}),
\end{eqnarray}

Thus, from (\ref{pe2}), if  $\textbf{M}$ does not have $\mathfrak{l}\mathfrak{g}$-property, then we must have \begin{eqnarray}\label{pe3} \lambda_1(M_{11}\sharp M_{22})<s_1(M_{12}).\end{eqnarray}    
Inequalities (\ref{pe2}) and (\ref{pe3}) together yield $$\lambda_2(M_{11}\sharp M_{22})\ge s_2(M_{12}).$$ Taking into account that  $\lambda_2(M_{11}\sharp M_{22})\le \frac{1}{2}\lambda_2(M_{11}+M_{22})$ gives
\begin{eqnarray}\label{pe4} s_2(M_{12})\le \frac{1}{2}\lambda_2(M_{11}+M_{22}).
\end{eqnarray}  Thus (\ref{pe1}) and (\ref{pe4}) together indicate that $\textbf{M}$ has $\mathfrak{a}$-property.

Assume now that $\textbf{M}$ does not have $\mathfrak{a}$-property,  then in view of (\ref{pe1}), $$s_2(M_{12})>\frac{1}{2}\lambda_2(M_{11}+M_{22})\ge \lambda_2(M_{11}\sharp M_{22}).$$ Therefore, from (\ref{pe2}) we conclude that \begin{eqnarray}\label{pe5} s_1(M_{12})\le \lambda_1(M_{11}\sharp M_{22}). \end{eqnarray}   Thus (\ref{pe2}) and (\ref{pe5}) together indicates that $\textbf{M}$ has   $\mathfrak{l}\mathfrak{g}$-property
 \end{proof}

The following numerical example shows that Proposition \ref{p1} needs not be true if each block of  $\textbf{M}$ is $3\times 3$. In other words, there are matrices that have only  $\mathfrak{l}\mathfrak{a}$-property but no other three. 

\begin{eg} Take $M_{11}=\begin{bmatrix}1.7353 &  -0.2433 &   1.7146
\\   -0.2433  &  1.6438 &   0.7227\\
    1.7146 &   0.7227 &   6.6795  \end{bmatrix}^2$,
\\ $M_{22}=\begin{bmatrix} 2.7266 &  -1.3731 &  -0.0930\\
   -1.3731  &  2.3151   & 0.0859\\
   -0.0930  &  0.0859   & 0.7646 \end{bmatrix}^2$ and $X=\begin{bmatrix} -0.0445 &  -0.9170 &  -0.3964\\
    0.6927  & -0.3142  &  0.6492\\
   -0.7198  & -0.2457  &  0.6492\end{bmatrix}$. As $X$ is a contraction, we may take $M_{12}=M_{11}^{1/2}XM_{22}^{1/2}$ so that  $\textbf{M}$  is positive  semidefinite (see \cite[p. 207]{HJ91}). A calculation using Matlab shows 
   \begin{eqnarray*}
   \lambda(M_{11}\sharp M_{22})&=&\{7.2176, 5.5156, 1.0415\};\\
   s(M_{12})&=&\{8.7154,
   3.2243,
   1.4755\};\\
   \lambda((M_{11}+M_{22})/2)&=&\{26.9680,  9.2207, 1.0879\}.
   \end{eqnarray*}  So $\textbf{M}$ has  $\mathfrak{l}\mathfrak{a}$-property.  But it neither has $\mathfrak{l}\mathfrak{g}$-property nor $\mathfrak{a}$-property.
\end{eg}

 \begin{rem}
 	 In the computation of matrix geometric mean, we used the programme developed by Bini and Iannazzo \cite{BI}.
 \end{rem}
 
The next example shows that   $\mathfrak{g}$-property could be more strict than the other three.
\begin{eg} Take $M_{11}=\begin{bmatrix}  1   &  0
\\   0  &   2\end{bmatrix}^2$, $M_{22}=\begin{bmatrix} 2&-1\\-1& 1\end{bmatrix}^2$ and $X=\frac{\sqrt{2}}{2}\begin{bmatrix} -1 &   1 \\ 1  &  1\end{bmatrix}$.  As $X$ is a unitary, we may take $M_{12}=M_{11}^{1/2}XM_{22}^{1/2}$ so that  $\textbf{M}$  is positive  semidefinite. In this case, $|\det M_{12}|=\sqrt{\det M_{11}M_{22}}=\det M_{11}\sharp M_{22}$. A calculation using Matlab shows   
 \begin{eqnarray*}\lambda(M_{11}\sharp M_{22})&=&\{ 3.0760,
    0.6502\}; \\ s(M_{12})&=&\{ 2.8284,
    0.7071 \}; \\ \lambda((M_{11}+M_{22})/2)&=&\{4.5000,  1.5000 \}.  \end{eqnarray*}  So $\textbf{M}$ has  $\mathfrak{l}\mathfrak{g}$-property (one only needs to check that $\lambda_1(M_{11}\sharp M_{22})> s_1(M_{12})$) and  $\mathfrak{a}$-property.  But it does not have  $\mathfrak{g}$-property.\end{eg}

There are examples that  $\textbf{M}$ has  $\mathfrak{l}\mathfrak{g}$-property but no  $\mathfrak{a}$-property; see Section 3. Examples that $\textbf{M}$ has $\mathfrak{a}$-property but no   $\mathfrak{l}\mathfrak{g}$-property are considered in Section 4. 
We may use a Venn diagram to illustrate  relations between these four notions. 

 ~
 
	 	 \begin{center} 
	 	 	\includegraphics[height=5cm, width=9cm]{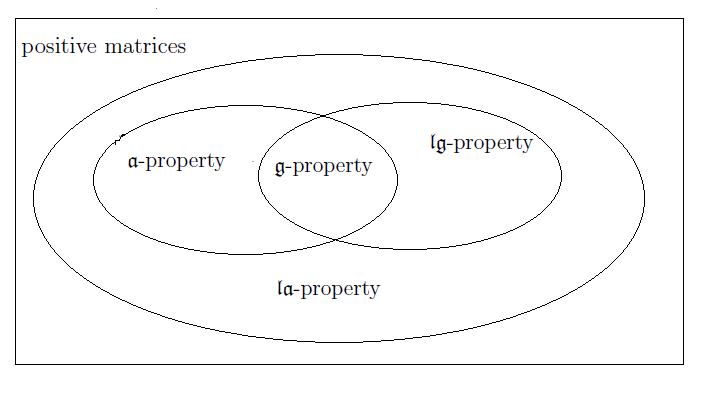}
	 	 \end{center}

\section{PPT matrices}\label{s3}
In this section, we present initial incentive for the investigation in this paper. 

 A positive  semidefinite matrix  $\textbf{M}=\begin{bmatrix} M_{11} & M_{12} \\ M_{12}^* & M_{22}\end{bmatrix}$  is said to have  positive partial transpose (PPT) if its partial transpose $\textbf{M}^\tau=\begin{bmatrix} M_{11} & M_{12}^*\\ M_{12} & M_{22}\end{bmatrix}$ is also positive  semidefinite. The partial transpose is an intriguing operation, it is different from the conventional transpose in many aspects, for example, $(\textbf{M}^\tau)^2\ne (\textbf{M}^2)^\tau$ in general.

The following theorem was proved in \cite{Lin15}. 
 \begin{thm}\label{thm1}  Let  $\textbf{M}=\begin{bmatrix} M_{11} & M_{12} \\ M_{12}^* & M_{22}\end{bmatrix}$ be PPT.   Then \begin{eqnarray*} \prod_{j=1}^ks_j(M_{12})\le \prod_{j=1}^k \lambda_j(M_{11}\sharp M_{22}), \qquad j=1, \ldots, n.
\end{eqnarray*}  \end{thm}
In other words, PPT matrices have  $\mathfrak{l}\mathfrak{g}$-property. It is noteworthy that not all PPT matrices have $\mathfrak{a}$-property. We present two different examples. 

 \begin{eg}  The matrix  $$\begin{bmatrix} A^2+B^2 & AB+BA \\ AB+BA&  A^2+B^2 \end{bmatrix}$$ is PPT  whenever $A, B$ are $n\times n$ Hermitian matrices.
 The $\mathfrak{a}$-property in this example is equivalent to  \begin{eqnarray}\label{e31} s_j(AB+BA)\le  \lambda_j( A^2+B^2), \qquad j=1, \ldots, n.\end{eqnarray}  However, (\ref{e31}) fails in general; see \cite[p. 2182]{BK08}. This example has been given in \cite{Lin15}. \end{eg}

 \begin{eg} It is easy to see that if $A, B$ are $n\times n$ positive  semidefinite   matrices,  then  $$\begin{bmatrix} A+B & A-B \\ A-B&  A+B \end{bmatrix}
 $$ is PPT.  The $\mathfrak{a}$-property in this example is equivalent to  \begin{eqnarray}\label{e32} s_j(A-B)\le  \lambda_j(A+B), \qquad j=1, \ldots, n.\end{eqnarray}   Again, the concrete example in \cite[p. 2179]{BK08} shows   (\ref{e32})  fails in general. \end{eg}

  In the sequel, the norm $\|\cdot\|$ stands for the usual spectral norm, i.e., $\|\cdot\|=s_1(\cdot)$.  Ando proved the following norm inequality. 
  
   \begin{thm}\label{thm2} \cite[Theorem 3.3]{And16} Let  $\textbf{M}=\begin{bmatrix} M_{11} & M_{12} \\ M_{12}^* & M_{22}\end{bmatrix}$ be PPT. Then \begin{eqnarray*} \|M_{12}\|\le \|M_{11}\sharp M_{22}\|.
   	\end{eqnarray*}\end{thm}
  
  Though Theorem \ref{thm2} looks weaker in form than Theorem \ref{thm1}, we use a standard approach to show they are essentially equivalent. 
  
  If $\wedge^k(X)$, $1\le k\le n$, denotes the $k$-th antisymmetric tensor power \cite[p. 16]{Bha97} of an $n\times n$ matrix $X$, then $\|\wedge^k(X)\|=s_1(\wedge^k(X))=\prod_{j=1}^ks_j(X)$. Note that $$\wedge^k(M_{11}\sharp M_{22})=\Big(\wedge^k(M_{11})\Big)\sharp \Big(\wedge^k(M_{22})\Big).$$
  Thus, to show that Theorem \ref{thm2} implies Theorem \ref{thm1}, it suffices to show that if $\begin{bmatrix} M_{11} & M_{12}\\ M_{12}^* & M_{22}\end{bmatrix}$ is PPT, then so is $\begin{bmatrix}\wedge^k (M_{11}) & \wedge^k(M_{12})\\ \wedge^k(M_{12}^*) &  \wedge^k(M_{22})\end{bmatrix}$. Without loss of generality, we assume $M_{22}$ is positive  definite (the general case follows by a standard continuity argument). Consider the Schur complement
  \begin{eqnarray*}&&\wedge^k (M_{11})-\wedge^k(M_{12})(\wedge^k(M_{22}))^{-1} \wedge^k(M_{12}^*)\\&=& \wedge^k (M_{11})-\wedge^k(M_{12})\wedge^k(M_{22}^{-1}) \wedge^k(M_{12}^*)
  \\&=& \wedge^k (M_{11})-\wedge^k(M_{12}M_{22}^{-1}M_{12}^*)\\&\ge& \wedge^k (M_{11} -M_{12}M_{22}^{-1}M_{12}^*)\ge 0,
  \end{eqnarray*}
  in which the first inequality is by \cite[(4.20), p. 114]{Bha07}. Similarly, \begin{eqnarray*}&&\wedge^k (M_{11})-\wedge^k(M_{12}^*)(\wedge^k(M_{22}))^{-1} \wedge^k(M_{12})  \\&\ge& \wedge^k (M_{11} -M_{12}^*M_{22}^{-1}M_{12})\ge 0,
  \end{eqnarray*} 
  as desired.

  We remark that a simple proof of    Theorem \ref{thm1} has appeared in 	\cite{Lee15}. 
   
 	A typical example of PPT matrix is the Hua matrix which has the form   
  $$\begin{bmatrix} (I_n-A^*A)^{-1}&(I_n-B^*A)^{-1}\\
 	(I_n-A^*B)^{-1}& (I_n-B^*B)^{-1}
 	\end{bmatrix},$$  where $A, B$ are $m\times n$ strictly contractive  matrices. 
  So the Hua matrix has   $\mathfrak{l}\mathfrak{g}$-property. 	In \cite[Theorem 3.3]{Lin15}, we proved that it has $\mathfrak{a}$-property. Later, we used a simpler argument to show that the Hua matrix  has $\mathfrak{g}$-property; see \cite[Theorem 3.2]{Lin16b}.

In the next two examples, we assume that  $\begin{bmatrix}A & X  \\ X^*  &  B\end{bmatrix}$, where $A, X, B$ are $n\times n$, is positive  semidefinite. The trace of a square matrix $X$ is denoted by $\tr X$. 

\begin{eg}\label{egprob1} It is known that $$\begin{bmatrix}\Phi(A) & \Phi(X)  \\ \Phi(X^*)  &  \Phi(B)\end{bmatrix},$$
	where $\Phi:   X\mapsto X+(\tr X)I_n$, is PPT (see \cite{Lin14}). So the matrix $\begin{bmatrix}\Phi(A) & \Phi(X)  \\ \Phi(X^*)  &  \Phi(B)\end{bmatrix}$ has $\mathfrak{l}\mathfrak{g}$-property.
	
	It is recently proved \cite{Lin16} that the matrix $\begin{bmatrix}\Phi(A) & \Phi(X)  \\ \Phi(X^*)  &  \Phi(B)\end{bmatrix}$ has $\mathfrak{a}$-property, namely,   
	\begin{eqnarray*} 
	2s_j\Big(\Phi(X)\Big)\le s_j\Big(\Phi(A)+\Phi(B)\Big), \qquad j=1, \ldots, n.
	\end{eqnarray*}  
	Numerical experiments suggest that the matrix $\begin{bmatrix}\Phi(A) & \Phi(X)  \\ \Phi(X^*)  &  \Phi(B)\end{bmatrix}$ has $\mathfrak{g}$-property, which we haven't been able to prove yet. 	 
\end{eg} 

\begin{eg}\label{egprob2} If we consider the map $\Psi:    X\mapsto 2(\tr X)I_n-X$, then using the approach in \cite{Lin14} we can show that the matrix
$$\begin{bmatrix}\Psi(A) & \Psi(X)  \\ \Psi(X^*)  &  \Psi(B)\end{bmatrix}$$ is PPT. 
 
Though there are strong numerical evidence suggesting that this block matrix also has  $\mathfrak{g}$-property, yet we have not even been able to show that it has $\mathfrak{a}$-property. \end{eg}

\section{Non-PPT matrices}\label{s4}
Consider the positive  semidefinite  matrix $\textbf{M}=\begin{bmatrix} M_{11} & M_{12} \\ M_{12}^* & M_{22}\end{bmatrix}$. Assume further that the off diagonal block $M_{12}$ is unitary. 
  It is easy to see that the block matrix $\textbf{M}$  is not PPT in general. The next proposition says that under the extra unitary assumption on $M_{12}$, the   matrix $\textbf{M}$  has
 $\mathfrak{l}\mathfrak{g}$-property.

\begin{prop} If $M_{12}$ in the positive  semidefinite  matrix  $\textbf{M}=\begin{bmatrix} M_{11} & M_{12} \\ M_{12}^* & M_{22}\end{bmatrix}$ is unitary, then  \begin{eqnarray*}
1= \prod_{j=1}^ks_j(M_{12}) \le   \prod_{j=1}^k\lambda_j(M_{11}\sharp M_{22}), \qquad k=1, \ldots, n.
\end{eqnarray*}  \end{prop}
\begin{proof} As $\wedge^k(M_{12})$, $1\le k\le n$, is again unitary, similar to the argument in Section \ref{s3},   the required inequality is equivalent to 
\begin{eqnarray}\label{e41}\|M_{11}\sharp M_{22}\|\ge 1. \end{eqnarray}
	Inequality (\ref{e41}) is due to Ando \cite[Theorem 3.5]{And16}. 
  We include a proof for completeness.  First of all, we notice that $M_{11}$ and $M_{22}$ are nonsingular, as we may write  $M_{12}=M_{11}^{1/2}CM_{22}^{1/2}$ for some contraction $C$ (see \cite[p. 207]{HJ91}).   
  
    We need to prove that
 \begin{eqnarray*}  M_{11}\ge M_{12}M_{22}^{-1}M_{12}^* \Longrightarrow \|M_{11}\sharp M_{22}\|\ge 1.
\end{eqnarray*}
Assume otherwise that $\|M_{11}\sharp M_{22}\|<1$, i.e., $M_{11}\sharp M_{22}<I_n$, then due to the monotoncity of geometric mean $$X:=(M_{12}M_{22}^{-1}M_{12}^*)\sharp M_{22}<I_n$$ and so $\|X\|<1$. Moreoever,  $$M_{12}M_{22}^{-1}M_{12}^*=XM_{22}^{-1}X.$$ Taking norms on both sides gives $$\|M_{22}^{-1}\|=\|M_{12}M_{22}^{-1}M_{12}^*\|=\|XM_{22}^{-1}X\|\le \|X\|^2\|M_{22}^{-1}\|<\|M_{22}^{-1}\|.$$
 A contradiction.  \end{proof}

 To see that the positive  semidefinite  matrix $\textbf{M}=\begin{bmatrix} M_{11} & M_{12}\\ M_{12}^* & M_{22}\end{bmatrix}$ with $M_{12}$ unitary does not have  $\mathfrak{a}$-property in general, consider a special case  $M_{22}=M_{12}^*M_{11}^{-1}M_{12}$. Then  $\mathfrak{a}$-property in this case is equivalent to
 \begin{eqnarray}\label{e42}\lambda_j(M_{11}+M_{12}^*M_{11}^{-1}M_{12})\ge 2, \qquad j=1, \ldots, n.\end{eqnarray}
Take  $M_{11}=\begin{bmatrix} 1 & 0\\ 0 & 2\end{bmatrix}$, $M_{12}=\begin{bmatrix} 0 & 1\\ 1 & 0\end{bmatrix}$. Then
 $$M_{11}+M_{12}^*M_{11}^{-1}M_{12}=\begin{bmatrix} 3/2 & 0\\ 0 & 3\end{bmatrix}.$$
And so, $$\lambda_2(M_{11}+M_{12}^*M_{11}^{-1}M_{12})=3/2<2s_2(M_{12})=2,$$
violating (\ref{e42}).

\vspace{0.2in}

Now we  present two examples about  a positive  semidefinite  matrix that has   $\mathfrak{a}$-property but no   $\mathfrak{l}\mathfrak{g}$-property. 

\begin{eg}   Bhatia and Kittaneh \cite{BK90} proved that if $A, B$  are $n\times n$ positive  semidefinite  matrices, then
	\begin{eqnarray*} 2s_j(AB)\le \lambda_j(A^2+B^2), \qquad j=1, \ldots, n.
	\end{eqnarray*} 
	 This in particular says that the matrix
	 \begin{eqnarray*} \begin{bmatrix} A^2 &  AB\\ BA & B^2\end{bmatrix}
	 \end{eqnarray*} has  $\mathfrak{a}$-property. Now we explain that the matrix does not have  $\mathfrak{l}\mathfrak{g}$-property. It sufficies to show that
	\begin{eqnarray}\label{e43} \|AB\|\le \|A^2\sharp B^2\|. \end{eqnarray}
	 fails in general. 
	 
Indeed, the correct result is that the inequality sign in (\ref{e43}) should be reversed.   
	 In \cite{AH94},  Ando and Hiai proved   
	 \begin{eqnarray*} \|A^2\sharp B^2\|\le\|A \sharp B\|^2. \end{eqnarray*} Combinging with 
	  \begin{eqnarray*} \|A \sharp B\|^2\le \|A^{1/2} B^{1/2}\|^2=\lambda_1(AB)\le \|AB\| \end{eqnarray*}
gives 
\begin{eqnarray}\label{e44}  \|AB\|\ge \|A^2\sharp B^2\|.  \end{eqnarray}
 In particular, if $A, B$ do not commute, then the inequalities in (\ref{e44}) are strict.	 
\end{eg}

 \begin{eg}
Let $A, B$ be  $n\times n$ positive  semidefinite  matrices. We consider
$$ \begin{bmatrix} \|B\|A & AB  \\ BA  & \|A\|B \end{bmatrix}.$$

This matrix is positive  semidefinite, for $\|B\|A$ is positive  semidefinite  and the Schur complement
\begin{eqnarray*}\|B\|A-AB(\|A\|B)^{-1}BA&=&\|B\|A-\frac{1}{\|A\|}ABA\\&\ge &\|B\|\left(A-\frac{1}{\|A\|}A^2\right) 
\end{eqnarray*}   is positive  semidefinite. 

The matrix  has  $\mathfrak{a}$-property, which is proved in Proposition \ref{p4}. However, the matrix does not have  $\mathfrak{l}\mathfrak{g}$-property as we have a simple numerical example. 
Take $$A=\begin{bmatrix}1.7  &  1.3\\    1.3 &    1   \end{bmatrix}, \qquad B=\begin{bmatrix}2.2 &   -1.5\\
-1.5  &   1.1 \end{bmatrix}.$$ Then a   calculation gives 
$$\sqrt{\|A\|\|B\|}\|A\sharp B\|\approx 1.2055<\|AB\|\approx  2.6515.$$
\end{eg}

\begin{prop}\label{p4}  Let $A, B$ be  $n\times n$ positive  semidefinite  matrices. Then \begin{eqnarray*}2s_j(AB)\le \lambda_j(\|B\|A+\|A\|B), \qquad   j=1, \ldots , n.
\end{eqnarray*}\end{prop}
\begin{proof} The Bhatia-Kittaneh-Drury inequality \cite{Dru12} says that if $X, Y$ are  $n\times n$ positive  semidefinite  matrices, then 
	$$2\sqrt{s_j(XY)}\le \lambda_j(X+Y), \qquad   j=1, \ldots , n.$$ 
	This implies \begin{eqnarray*} \lambda_j(\|B\|A+\|A\|B)\ge 2\sqrt{s_j(\|A\|\|B\|AB)}=2\sqrt{\|A\|\|B\|s_j(AB)}.
\end{eqnarray*} But it is clear that $s_j(AB)\le \|A\|\|B\|$ for $j=1, \ldots , n$. So the required inequality is confirmed. \end{proof}

Finally, we present a simple non-PPT matrix that has $\mathfrak{g}$-property.
 
\begin{eg} Let $A$ be any  $n\times n$ positive  semidefinite  matrix. The positive   semidefinite  matrix \begin{eqnarray*} \begin{bmatrix} I &  A  \\ A^*  & A^*A\end{bmatrix}
\end{eqnarray*} is not PPT in general, but it has $\mathfrak{g}$-property. This is because
$$s_j(A)=\lambda_j(|A|)=\lambda_j(I\sharp A^*A)$$
for  $j=1, \ldots , n$. \end{eg} 

\section{Concluding remarks}
We point out some closely related questions for future considerations. 

\noindent $\dag$.  Besides the challenging problems described in Example \ref{egprob1} and  Example \ref{egprob2}, other maps could be considered/constructed to meet these four properties.

\noindent $\ddag$. A generic criterion for $\mathfrak{l}\mathfrak{g}$-property is the PPT condition (Theorem \ref{thm1}). It would be of great interest to know similar conditions for other three properties. 

\noindent $\S$. One may add two new relations to Definition 1.1. More precisely,  	\begin{enumerate}  \item[(v)]  $\textbf{M}$ is said to have $\mathfrak{m}\mathfrak{a}$-property if\begin{eqnarray*}
		\sum_{j=1}^k2s_j(M_{12})\le   \sum_{j=1}^k\lambda_j(M_{11}+M_{22}), \qquad k=1, \ldots, n.
	\end{eqnarray*}
	
	\item[(vi)]  	$\textbf{M}$ is said to have  $\mathfrak{m}\mathfrak{g}$-property if
	\begin{eqnarray*}
		\sum_{j=1}^ks_j(M_{12})\le   \sum_{j=1}^k\lambda_j(M_{11}\sharp M_{22}), \qquad k=1, \ldots, n.
	\end{eqnarray*}	 
\end{enumerate}  
	 
 This of course deserves further investigation. 

\subsection*{Acknowledgments} {\small The work is supported in part by a grant from NNSFC. }

\end{document}